\documentclass[11pt]{amsart}
\usepackage{amsmath}

\allowdisplaybreaks

\usepackage{amsfonts}

\usepackage{amsthm,amssymb}

\renewcommand{\baselinestretch}{\baselinestretch}
\renewcommand{\baselinestretch}{1.2}

\numberwithin{equation}{section}
\newtheorem{lemma}[equation]{Lemma}
\newtheorem{proposition}[equation]{Proposition}
\newtheorem{theorem}[equation]{Theorem}
\newtheorem{corollary}[equation]{Corollary}
\theoremstyle{definition}
\newtheorem{example}[equation]{Example}%[section]
\newtheorem{remark}[equation]{Remark}
\providecommand{\bysame}{\leavevmode\hbox
	to3em{\hrulefill}\thinspace}

\usepackage[all,cmtip]{xy}
\usepackage{enumitem}
\usepackage{tikz-cd}
\usepackage{hyperref}
\usepackage{mathrsfs}
\usepackage{tikz}
\usepackage[english]{babel}
\usepackage{graphicx}
\usepackage{centernot}
\usepackage{mathtools}
\newlist{inlineroman}{enumerate*}{1}
\setlist[inlineroman]{itemjoin*={{, and }},afterlabel=~,label=\roman*.}

 \numberwithin{dummy}{section}

\newenvironment{red}{\relax\color{red}}{\relax}
\newenvironment{blue}{\relax\color{blue}}{\hspace*{.5ex}\relax}
\newcommand{\ber}{\begin{red}}
\newcommand{\er}{\end{red}}
\newcommand{\beb}{\begin{blue}}
\newcommand{\eb}{\end{blue}}

\begin{document}
	
\newcommand{\pr}{\partial}
\newcommand{\nl}{\vskip 1pc}
\newcommand{\co}{\mbox{co}}
\newcommand{\ol}{\overline}
\newcommand{\om}{\Omega}
\newcommand{\ra}{\rightarrow}
\newcommand{\epsil}{\varepsilon}
\makeatletter
\newcommand*{\rom}[1]{\expandafter\@slowromancap\romannumeral #1@}
\newcommand{\powerset}{\raisebox{.15\baselineskip}{\Large\ensuremath{\wp}}}
\newcounter{counter}       
\newcommand{\upperRomannumeral}[1]{\setcounter{counter}{#1}\Roman{counter}}
\newcommand{\lowerromannumeral}[1]{\setcounter{counter}{#1}\roman{counter}}
\newcommand*\circled[1]{\tikz[baseline=(char.base)]{
		\node[shape=circle,draw,inner sep=1pt] (char) {#1};}}
\makeatother
\title{Rationality and $p$-adic properties of reduced forms of half-integral weight}

\author[S.~H. Choi, C.~H. Kim, Y.-W. Kwon and K.-H. Lee]{Suh Hyun Choi, Chang Heon Kim, Yeong-Wook Kwon and Kyu-Hwan Lee}

\address{Department of
Mathematics, University of Connecticut, Storrs, CT 06269, U.S.A.}
\email{suhhyun.choi@gmail.com}

\address{Department of Mathematics, Sungkyunkwan University, Suwon 16419, South Korea}
\email{chhkim@skku.edu}
\thanks{The second named author was supported by the National Research Foundation of Korea(NRF) grant funded by the Korea government(MSIT) (NRF-2015R1D1A1A01057428 and 2016R1A5A1008055)}

\address{Department of Mathematics, Sungkyunkwan University, Suwon 16419, South Korea}
\email{pronesis196884@gmail.com}

\address{Department of
Mathematics, University of Connecticut, Storrs, CT 06269, U.S.A.}
\email{khlee@math.uconn.edu}
\thanks{The fourth named author was partially supported by a grant from the Simons Foundation (\#318706).}



\begin{abstract}  
In this paper we study special bases of certain spaces of half-integral weight weakly holomorphic modular forms. We establish a criterion for the integrality of Fourier coefficients of such bases. By using recursive relations between Hecke operators, we derive relations of Fourier coefficients of each basis element and  obtain congruences of the Fourier coefficients, which extend known congruences for traces of singular moduli. 
\end{abstract}
\maketitle
%\thispagestyle{empty}
%\tableofcontents
% \clearpage

\section{Introduction and Statement of Results}\label{intro}

Let $N$ be a positive integer. For an odd integer $k$, we denote by $M^{!+\cdots +}_{k/2}(N)$ the space of weakly holomorphic modular forms of weight $k/2$ on $\Gamma_{0}(4N)$ whose $n^{\text{th}}$ Fourier coefficient vanishes unless $(-1)^{(k-1)/2}\,n$ is a square modulo $4N$. For the moment, we assume that $N$ is contained in the set
$$\mathfrak{S}=\{1,2,3,5,7,11,13,17,19,23,29,31,41,47,59,71\}.$$
Then the group $\Gamma_{0}^{*}(N)$, which is the group generated by $\Gamma_{0}(N)$ and all Atkin--Lehner involutions $W_{e}$ for $e\parallel N$, has genus $0$. 
From the correspondence between Jacobi forms and half-integral weight forms (cf. \cite[Theorem 5.6]{EiZa1}), 
we see that for any $D\in\mathbb{Z}_{>0}$ with $D\equiv\square\pmod{4N}$, there is a unique modular form $g_{D,N}\in M^{!+\cdots +}_{3/2}(N)$ having a Fourier expansion of the form 
$$g_{D,N}(\tau)=q^{-D}+\sum_{d\geq 0}B^{(N)}(D,d)\, q^{d}\quad(q=e^{2\pi i\tau},~\tau\in\mathbb{H}).$$
Here, $\mathbb{H}$ denotes the complex upper half plane.

Let $\ell$ be a prime with $\ell\nmid 4N$. Then the Hecke operator $T_{k/2,4N}(\ell^{2})$, originally defined on the space of weakly holomorphic modular forms of weight $k/2$ on $\Gamma_{0}(4N)$, acts on $M_{k/2}^{!+\cdots+}(N)$. We define $T_{k/2,4N}(\ell^{2n})$ for $n\geq 2$ recursively by
$$T_{k/2,4N}(\ell^{2n}):=T_{k/2,4N}(\ell^{2n-2})T_{k/2,4N}(\ell^{2})-\ell^{k-2}T_{k/2,N}(\ell^{2n-4}).$$
For any positive integer $m$ with $\gcd{(m,4N)}=1$, 
 define $T_{k/2,4N}(m^{2})$ multiplicatively and set
$$g_{D,N}^{(m)}:=g_{D,N}\mid T_{3/2,4N}(m^{2}).$$
We denote by $B^{(N)}_{m}(D,d)$ the $d^{\text{th}}$ Fourier coefficient of $g_{D,N}^{(m)}(\tau)$:
$$g_{D,N}^{(m)}(\tau)=\mbox{(principal part)}+\sum_{d\geq 0}B^{(N)}_{m}(D,d)\, q^{d}.$$

By the works of Zagier \cite{Zag2} and Kim \cite{Kim2}, the coefficients $B^{(N)}_{m}(D,d)$ can be interpreted as traces of CM values of certain modular functions (or traces of singular moduli).
Remarkably, the coefficients $B^{(N)}_{m}(D,d)$ show many congruence properties, and many authors studied them. In 2005, Ahlgren and Ono \cite{AhO1} showed that if $p\nmid m$ is an odd prime and $\left(\frac{-d}{p}\right)=1$, then
$$B^{(1)}_{m}(1,p^{2}d)\equiv 0\pmod{p}.$$
Edixhoven \cite{Edix1} used the $p$-adic geometry of modular curves to show that, for any $m$ and any $d$ with $\left(\frac{-d}{p}\right)=1$, we have
\begin{equation*}
B^{(1)}_{m}(1,p^{2n}d)\equiv 0\pmod{p^n}.
\end{equation*}
When $p$ is an odd prime, Jenkins \cite{Jen1} obtained a recursive formula for $B^{(1)}_{1}(D,p^{2n}d)$ in terms of $B^{(1)}_{1}(D,p^{2k}d)$ with $k<n$. As a corollary he proved that if $\left(\frac{-d}{p}\right)=\left(\frac{D}{p}\right)\neq 0$, then we have
\begin{equation*}\label{1.8}
B^{(1)}_{1}(D,p^{2n}d)=p^{n}B^{(1)}_{1}(p^{2n}D,d).
\end{equation*}
Guerzhoy \cite{Guer1} showed that if $D$ and $-d$ are fundamental discriminants with $\left(\frac{-d}{p}\right)=\left(\frac{D}{p}\right)$, then, for any $m$, we have
\begin{equation*}
B^{(1)}_{m}(D,p^{2n}d)=p^{n}B^{(1)}_{m}(p^{2n}D,d).
\end{equation*}
In 2012, Ahlgren \cite{Ahl1} proved a general theorem which implies the above results as special cases. On the other hand, Osburn \cite{Osb1} proved that if $d$ is a positive integer such that $-d$ is congruent to a square modulo $4N$ and if $p\neq N$ is an odd prime which splits in $\mathbb{Q}(\sqrt{-d})$, then
$$B^{(N)}_{1}(1,p^{2}d)\equiv 0\pmod{p}.$$
Jenkins \cite{Jen2} and Koo and Shin \cite{KS1} obtained the following generalization of Osburn's result: for a positive integer $d$ such that $-d\equiv\square\pmod{4N}$ and an odd prime $p\neq N$ which splits in $\mathbb{Q}(\sqrt{-d})$,
$$B^{(N)}_{1}(1,p^{2n}d)\equiv 0\pmod{p^{n}}$$
for all $n\geq 1$.

The purpose of this paper is to generalize all these congruences to more general modular forms. To be precise, from now on, we assume that $N\geq 1$ is odd and square-free. For an even Dirichlet character $\chi$ modulo $4N$, we denote by $M^{!}_{k/2}(4N,\chi)$ the space of weakly holomorphic modular forms of weight $k/2$ on $\Gamma_{0}(4N)$ with Nebentypus $\chi$. The subspace of holomorphic forms and that of cuspforms are denoted by $M_{k/2}(4N,\chi)$ and $S_{k/2}(4N,\chi)$ respectively.

Let $\mathcal{D}$ be a discriminant form of level $4N$ satisfying some additional conditions which will be given in Section \ref{sec-epsilon}. (For the basics on discriminant forms, see Section \ref{disc} below.) Then $\mathcal{D}$ determines an even Dirichlet character $\chi$ modulo $4N$ and a sign vector $\epsilon=(\epsilon_{p})_{p}$ over $p=2$ or $p\mid N$ with $\chi_{p}\neq 1$, where the character $\chi$ is decomposed into $p$-components: $\chi=\prod_{p}\chi_{p}$. Set $\chi'=\chi\left(\frac{4N}{\cdot}\right)$.

We define the associated modular form space $M_{k/2}^{!\epsilon}(N,\chi')$ to be the subspace of $M_{k/2}^{!}(4N,\chi')$ consisting of the forms $f\in M_{k/2}^{!}(4N,\chi')$ satisfying the so-called \emph{$\epsilon$-condition}, which will be defined in Section \ref{sec-epsilon}. We let $$M_{k/2}^{\epsilon}(N,\chi')=M_{k/2}^{!\epsilon}(N,\chi')\cap M_{k/2}(4N,\chi') \text{ and } S_{k/2}^{\epsilon}(N,\chi')=M_{k/2}^{!\epsilon}(N,\chi')\cap S_{k/2}(4N,\chi').$$ 

Let us give an example. Consider the following even lattice
$$L=\left\{\left(\begin{array}{cc}a & b/N \\ c & -a\end{array}\right):a,b,c\in\mathbb{Z}\right\},$$
with $Q(\alpha)=-N\det(\alpha)$ and $(\alpha,\beta)=N\mathrm{tr}(\alpha\beta)$. We denote by $L'$ the dual lattice of $L$. Then the space $M_{k/2}^{!\epsilon}(N,\chi')$ associated with the discriminant form $L'/L$ is exactly the same as the space $M_{k/2}^{!+\cdots+}(N)$. Hence the $\epsilon$-condition can be considered as a generalization of the Kohnen plus condition. %For details, see \textcolor{blue}{Section \ref{sec-epsilon}}.

Now we further assume that $\chi_{p}\neq 1$ for each $p\mid N$, so $\chi'=1$. In \cite{Zhang2}, Zhang defined a family of forms in $M_{k/2}^{!\epsilon}(N,1)$, called \emph{reduced forms}. (For the definition, see Section \ref{sec-epsilon}.) If a reduced form $f_m$ exists for some $m \in \mathbb Z$, it must be unique and $\chi_{p}(m)\neq-\epsilon_{p}$ for each $p\mid N$. The set of reduced modular forms forms a basis for $M_{k/2}^{!\epsilon}(N,1)$. When $k=3$ and $N\in\mathfrak{S}$, the reduced form $f_{-D}$ exists for each $D>0$ which is a square modulo $4N$ (cf. Proposition \ref{prop-zhang2} below). In fact, $s(-D)f_{-D}=g_{D,N}$ for every $D$ where $s(-D)$ is a scaling constant. Thus the reduced forms are natural generalizations of the forms $g_{D,N}$. 

In order to generalize the congruences mentioned above to reduced forms, we first need to check integrality of the Fourier coefficients of reduced forms. We establish the following proposition which allows us to check whether a fixed reduced form has integer Fourier coefficients.

\begin{proposition} \label{prop-1.1}
Let $k$ be an odd integer. Assume that $f=\sum_{n}a(n)q^{n}\in M_{k/2}^{!\epsilon}(N,\chi')\cap\mathbb{Q}(\!(q)\!)$ with bounded denominator, and that $a(n)\neq 0$ for some $n<0$. Furthermore, let $k'$ be the smallest positive integer which satisfies $k'\geq |\mathrm{ord}_{\infty}(f)|/4N$ and $k+12k'> 0$. If $a(n)\in\mathbb{Z}$ for $n\leq\mathrm{ord}_{\infty}(f)+\frac{k+12k'}{12}[\mathrm{SL}_{2}(\mathbb{Z}):\Gamma_{0}(4N)]$, then $a(n)\in\mathbb{Z}$ for all $n$.
\end{proposition}

Let $\mathcal{D}^{*}$ be the dual discriminant form of $\mathcal{D}$. It is known that the corresponding data to $\mathcal{D}^{*}$ is $(4N,\chi',\epsilon^{*})$ with $\epsilon_{p}^{*}=\chi_{p}(-1)\epsilon_{p}$. Denote by $M_{k/2}^{\epsilon^{*}}(N,\chi')$ the space of modular forms associated to $\mathcal{D}^{*}$. We denote by $a(m,n)$ the $n^{\mathrm{th}}$ Fourier coefficient of the reduced form $f_{m}$. 
%For convenience, if $m$ is an integer for which $f_{m}$ does not exist, then we let $a(m,n)=0$ for all $n$. 
We prove the following theorem which turns the integrality problem for reduced forms into checking finitely many of them.

\begin{theorem} \label{thm-1.2}
Let $m_{\epsilon}=\max\{m : f_{m}^{*}\in M_{2-k/2}^{\epsilon^{*}}(N,\chi')~\textnormal{exists}\}$. Assume that for all $n\in\mathbb{Z}$ and $m\geq -4N-m_{\epsilon}$, we have $s(m)a(m,n)\in\mathbb{Z}$. Then $s(m)a(m,n)\in\mathbb{Z}$ for all $m,n\in\mathbb{Z}$.
\end{theorem}

Therefore, to check the integrality of reduced forms, it suffices to show the integrality of a finite number of Fourier coefficients satisfying the conditions of both Proposition \ref{prop-1.1} and Theorem \ref{thm-1.2}. We give an example to illustrate this.

\begin{example} \label{ex-int}
We consider the space {$M_{1/2}^{!+\cdots +}(7,1)$}. Then we have $m_\epsilon=-1$. Define $$E_{k}(\tau)=1-\frac{2k}{B_{k}}\sum_{n=1}^{\infty}\sigma_{k-1}(n)q^{n}\qquad (2<k\in 2\mathbb{Z})$$ to be the normalized Eisenstein series,
and denote by $[\cdot,\cdot]_{n}$ ($n\geq 1$) the $n^{\text{th}}$ Rankin--Cohen bracket (cf. \cite[pp.53-58]{BrGeHaZa1}). Set
\[
\begin{array}{llll}
\mathrm{RC}_{1}=\frac{[\theta,E_{10}(28\tau)]_{1}}{\Delta(28\tau)}, & \mathrm{RC}_{2}=\frac{[\theta,E_{8}(28\tau)]_{2}}{\Delta(28\tau)}, & \mathrm{RC}_{3}=\frac{[\theta,E_{6}(28\tau)]_{3}}{\Delta(28\tau)}, & \mathrm{RC}_{4}=\frac{[\theta,E_{4}(28\tau)]_{4}}{\Delta(28\tau)}, \\ \mathrm{RC}_{5}=\frac{[\mathrm{RC}_{1},E_{10}(28\tau)]_{1}}{\Delta(28\tau)}, & \mathrm{RC}_{6}=\frac{[\mathrm{RC}_{1},E_{8}(28\tau)]_{2}}{\Delta(28\tau)}, &
\mathrm{RC}_{7}=\frac{[\mathrm{RC}_{1},E_{6}(28\tau)]_{3}}{\Delta(28\tau)}, & \mathrm{RC}_{8}=\frac{[\mathrm{RC}_{1},E_{4}(28\tau)]_{4}}{\Delta(28\tau)}, \\ \mathrm{RC}_{9}=\frac{[\mathrm{RC}_{2},E_{10}(28\tau)]_{1}}{\Delta(28\tau)}, &
\mathrm{RC}_{10}=\frac{[\mathrm{RC}_{2},E_{8}(28\tau)]_{2}}{\Delta(28\tau)}, & \mathrm{RC}_{11}=\frac{[\mathrm{RC}_{2},E_{6}(28\tau)]_{3}}{\Delta(28\tau)}, & \mathrm{RC}_{12}=\frac{[\mathrm{RC}_{2},E_{4}(28\tau)]_{4}}{\Delta(28\tau)}, \\
\mathrm{RC}_{13}=\frac{[\mathrm{RC}_{1},E_{10}(28\tau)]_{1}}{\Delta(28\tau)}, & \mathrm{RC}_{14}=\frac{[\mathrm{RC}_{3},E_{8}(28\tau)]_{2}}{\Delta(28\tau)}, & \mathrm{RC}_{15}=\frac{[\mathrm{RC}_{3},E_{6}(28\tau)]_{3}}{\Delta(28\tau)}, & \mathrm{RC}_{16}=\frac{[\mathrm{RC}_{3},E_{4}(28\tau)]_{4}}{\Delta(28\tau)}.\\
\end{array}
\]
In addition, we set
$$f=\tfrac{1}{5600}\mathrm{RC}_{1}+\tfrac{7}{103680}\mathrm{RC}_{2}+\tfrac{1}{80640}\mathrm{RC}_{3}+\tfrac{1}{705600}\mathrm{RC}_{4}-\tfrac{41687}{1800}\theta,$$
and define
\[ \mathrm{RC}_{17}=\tfrac{[f,E_{4}(28\tau)]_{4}}{\Delta(28\tau)}.\]
By taking linear combinations of these Rankin-Cohen brackets, we find
\begin{align*}
s(0)f_{0}&=1+2q+2q^4+2q^{9}+2q^{16}+\cdots,\\
s(-3)f_{-3}&=q^{-3}-3q-2q^{4}+6q^{8}+5q^{9}-10q^{16}+\cdots,\\
s(-7)f_{-7}&=q^{-7}-10q+4q^{4}+28q^{8}-24q^{9}+60q^{16}+\cdots,\\
s(-12)f_{-12}&=q^{-12}-10q-25q^{4}-6q^{8}+46q^{9}+152q^{16}+\cdots,\\
s(-19)f_{-19}&=q^{-19}-q-50q^{4}-50q^{8}-153q^{9}+798q^{16}+\cdots,\\
s(-20)f_{-20}&=q^{-20} - 22q + 26q^{4} - 180q^{8} - 78q^{9} - 338q^{16}+\cdots,\\
s(-24)f_{-24}&=q^{-24} - 2q - 28q^{4} + 225q^{8} - 450q^{9} - 2976q^{16}+\cdots,\\
s(-27)f_{-27}&=q^{-27} + 12q + 52q^{4} - 468q^{8} + 156q^{9} - 1300q^{16}+\cdots.   \end{align*}

For example, we obtain
\begin{align*}
f_{3}&=-\tfrac{92368453}{1197504000}\, \mathrm{RC}_{1}-\tfrac{1105849}{739031040}\, \mathrm{RC}_{2}-\tfrac{7775323}{804722688000}\, \mathrm{RC}_{3}+\tfrac{31109}{68584320000}\, \mathrm{RC}_{4}\\
&\quad-\tfrac{1}{49268736000}\, \mathrm{RC}_{7}+\tfrac{1}{862202880000}\, \mathrm{RC}_{8}-\tfrac{1}{86910050304000}\, \mathrm{RC}_{12}\\
&\quad+\tfrac{1}{216309458534400}\, \mathrm{RC}_{15}+\tfrac{83841213721}{1026432000}\, \theta\\
&=q^{-3} - 3q - 2q^{4} + 6q^{8} + 5q^{9} - 10q^{16}+\cdots.
\end{align*}
By Proposition \ref{prop-1.1}, the forms $s(0)f_{0},\ldots,s(-27)f_{-27}$ have integer Fourier coefficients. It follows from Theorem \ref{thm-1.2} that every reduced form in {$M_{1/2}^{!+\cdots+}(7,1)$} has integer Fourier coefficients.
\end{example}

Now we assume that, for any reduced form
$$f_{m}=\sum_{n}a(m,n)\, q^{n}\in M_{k/2}^{!\epsilon}(N,1),$$
the form $s(m)f_{m}$ has integer Fourier coefficients. Furthermore, let $k\geq 3$ be an odd integer and set $\lambda=(k-1)/2$. Then the reduced form $f_{m}\in M_{k/2}^{!\epsilon}(N,1)$ exists for every $m\in\mathbb{Z}_{<0}$ with $\chi_{p}(m)\neq -\epsilon_{p}$ for all $p\mid N$. We write
$$F_{m}(\tau)=s(m)f_{m}(\tau)=q^{m}+\sum\limits_{\substack{d\geq 0 \\ \chi_{p}(d)\neq -\epsilon_{p}~\textnormal{for all}~p\mid N}}B^{(N)}(m,d)\, q^{d}.$$
Note that the Hecke operator $T_{k/2,4N}(\ell^{2})$ acts on the space $M_{k/2}^{!\epsilon}(N,1)$ for each prime $\ell$ with $\gcd{(\ell,4N)}=1$. For any positive integer $t$ with $\gcd{(t,4N)}=1$, define
$$F_{m}^{(t)}:=F_{m}\mid T_{k/2,4N}(t^{2}).$$
Then we obtain the coefficients $B^{(N)}_{t}(m,d)$ from the equation
$$F_{m}^{(t)}(\tau)=\mbox{(principal part)}+\sum\limits_{\substack{d\geq 0,\\ \chi_{p}(d)\neq-\epsilon_{p}~\textnormal{for all}~p\mid N}}B^{(N)}_{t}(m,d)\, q^{d}.$$

We state our main theorem which describes various relations among the coefficients $B^{(N)}_{t}(m,d)$.

\begin{theorem}\label{thm-1.4}
We have the following:
\begin{enumerate}
\item[\textnormal{(\lowerromannumeral{1})}] $B^{(N)}_{t}(m,\ell^{2n+2}d)-\ell^{\lambda-1}\left(\frac{(-1)^{\lambda}m}{\ell}\right)B^{(N)}_{t}(m,\ell^{2n}d)=\ell^{(k-2)n}\left \{ B^{(N)}_{t}(\ell^{2n}m,\ell^{2}d)-B^{(N)}_{t}(\ell^{2n-2}m,d)\right \}$.
\item[\textnormal{(\lowerromannumeral{2})}] If $\ell\nmid d$, then 
\begin{align*}
\ell^{(k-2)n}&B^{(N)}_{t}(\ell^{2n}m,d)=B^{(N)}_{t}(m,\ell^{2n}d)\\
&\quad+\left[\left(\frac{(-1)^{\lambda}d}{\ell}\right)-\left(\frac{(-1)^{\lambda}m}{\ell}\right)\right]\cdot\sum_{k=1}^{n}\ell^{(\lambda-1)k}\left(\frac{(-1)^{\lambda}d}{\ell}\right)^{k-1}B^{(N)}_{t}(m,\ell^{2n-2k}d).
\end{align*}
\item[\textnormal{(\lowerromannumeral{3})}] If $\ell\parallel d$, then 
$$\ell^{(k-2)n}B^{(N)}_{t}(\ell^{2n}m,d)=B^{(N)}_{t}(m,\ell^{2n}d)-\ell^{\lambda-1}\left(\frac{(-1)^{\lambda}m}{\ell}\right)\cdot B^{(N)}_{t}(m,\ell^{2n-2}d).$$
\end{enumerate}
\end{theorem}

As a corollary, we obtain the following congruences:

\begin{corollary}\label{cor-1.5}
Assume that $S_{k/2}^{\epsilon}(N,1)=0$. 
\begin{enumerate}
\item If $\left(\frac{-d}{\ell}\right)=\left(\frac{-m}{\ell}\right)\neq 0$, or if $\ell\parallel d$ and $\ell\parallel m$, then for any positive integer $t$ with $(t,4N)=1$ and $n$, we have
$$B^{(N)}_{t}(m,\ell^{2n}d)=\ell^{(k-2)n}\, B^{(N)}_{t}(\ell^{2n}m,d)\equiv 0\pmod{\ell^{(k-2)n}}.$$
\item If $\chi_{p}(\ell d)\neq -\epsilon_{p}$ for all $p\mid N$, then for any positive integer $t$ with $(t,4N)=1$ and any $n\geq 1$, we get
$$B^{(N)}_{t}(m,\ell^{2n+1}d)\equiv\ell^{\lambda-1}\left(\frac{(-1)^{\lambda}m}{\ell}\right)B^{(N)}_{t}(m,\ell^{2n-1}d)\pmod{\ell^{(k-2)n}}.$$
\end{enumerate}
\end{corollary}

As for the condition in the above corollary, we remark that if
\begin{align*}
N\in\{n\mid n~\textnormal{is an odd } & \textnormal{square-free integer with}~1\leq n< 37\}\\
&\cup\{39,41,47,51,55,59,69,71,87,95,105,119\},
\end{align*}
then $S_{3/2}^{+\cdots +}(N,1)=0$. (See \cite[Table 4]{BrEhFr1}.)

We organize this paper as follows. In Section 2, we present preliminaries on discriminant forms and modular forms of half-integral weight. Also, we recall the definitions of $\epsilon$-condition and reduced forms. In Section 3, we prove Proposition \ref{prop-1.1} and Theorem \ref{thm-1.2}, and in Section 4 we prove Theorem \ref{thm-1.4} and Corollary \ref{cor-1.5}.

\section{Preliminaries}
In this section, we review the basics on discriminant forms, modular forms of half-integral weight and present a recent result of Zhang \cite{Zhang2}.

\subsection{Discriminant forms.} \label{disc} A discriminant form is a finite abelian group {$\mathcal{D}$} with a quadratic form {$Q:\mathcal{D}\rightarrow\mathbb{Q}/\mathbb{Z}$}, such that the symmetric bilinear form defined by $(\beta,\gamma)=Q(\beta+\gamma)-Q(\beta)-Q(\gamma)$ is nondegenerate, namely, the map {$\mathcal{D}\mapsto\mathrm{Hom}(\mathcal{D},\mathbb{Q}/\mathbb{Z})$} defined by $\gamma\mapsto(\gamma,\cdot)$ is an isomorphism. We define the level of a discriminant form {$\mathcal{D}$} to be the smallest positive integer $N$ such that $NQ(\gamma)=0$ for each {$\gamma\in\mathcal{D}$}. It is known that if $L$ is an even lattice then $L'/L$ is a discriminant form, where $L'$ is the dual lattice of $L$. Conversely, any discriminant form can be obtained in this way. The signature {$\mathrm{sign}(\mathcal{D})\in\mathbb{Z}/8\mathbb{Z}$} is defined to be the signature of $L$ modulo 8 for any even lattice $L$ such that {$L'/L=\mathcal{D}$}.

Every discriminant form can be decomposed uniquely into $p$-components {$\mathcal{D}=\oplus_{p}\mathcal{D}_{p}$}. Each $p$-component can be written as a direct sum of indecomposable Jordan $q$-components with $q$ powers of $p$. Such decompositions are not unique in general. We recall the possible indecomposable Jordan $q$-components as follows.

Let $p$ be an odd prime and $q>1$ be a power of $p$. The indecomposable Jordan components with exponent $q$ are denoted by $q^{\delta_{q}}$ with $\delta_{q}=\pm 1$. These discriminant forms both have level $q$.

If $q>1$ is a power of 2, there are also precisely two indecomposable \emph{even} Jordan components of exponent $q$, denoted by $q^{\delta_{q}2}$ with $\delta_{q}=\pm 1$. Such components have level $q$. There are also \emph{odd} indecomposable Jordan components, denoted by $q_{t}^{\pm 1}$ with $\pm 1=\left(\frac{2}{t}\right)$ for each $t\in (\mathbb{Z}/8\mathbb{Z})^{\times}$. These discriminant forms have level $2q$.

We call a discriminant form {$\mathcal{D}$} \emph{transitive} if  the action of {$\mathrm{Aut}(\mathcal{D})$} is transitive on the subset of elements of norm $n$ for any $n\in\mathbb{Q}/\mathbb{Z}$. By the classification of transitive forms, the level $N=\prod_{p}N_{p}$ of a transitive form {$\mathcal{D}=\oplus_{p}\mathcal{D}_{p}$} is of the following form: $N_{p}=1$ or $p$ for an odd prime $p$ and $N_{2}=1,2,4$ or 8. In other words, $N$ is the conductor of a quadratic Dirichlet character.

Let {$\mathcal{D}$} be a transitive discriminant form of odd signature $r$ and level $N$. Then {$\mathcal{D}$} determines an even quadratic character $\chi$ modulo $N$. Explicitly, it is given as follows: Decompose $\chi=\prod_{p}\chi_{p}$ into $p$-components. If $p$ is odd,
{\begin{equation*}
\chi_{p}(d)=
\begin{cases}
1, & \textnormal{if}~p\nmid|\mathcal{D}|~\textnormal{or}~p^{2}\mid|\mathcal{D}|,\\
\left(\frac{d}{p}\right), & \textnormal{otherwise}.
\end{cases}
\end{equation*}
When $2\mid |\mathcal{D}|$,
\begin{equation*}
\chi_{2}(d)=
\begin{cases}
1, & \textnormal{if}~\left(\frac{-1}{|\mathcal{D}|}\right)=+1~\textnormal{and}~\mathcal{D}_{2}=2_{\pm 3}^{+3},~2_{\pm 1}^{+1},\\
\left(\frac{-4}{d}\right), & \textnormal{if}~\left(\frac{-1}{|\mathcal{D}|}\right)=-1~\textnormal{and}~\mathcal{D}_{2}=2_{\pm 3}^{+3},~2_{\pm 1}^{+1},\\
\left(\frac{2}{d}\right), & \textnormal{if}~\left(\frac{-1}{|\mathcal{D}|}\right)=+1~\textnormal{and}~\mathcal{D}_{2}=4_{\pm 1}^{+1},~4_{\pm 3}^{-1},\\
\left(\frac{-2}{d}\right), & \textnormal{if}~\left(\frac{-1}{|\mathcal{D}|}\right)=-1~\textnormal{and}~\mathcal{D}_{2}=4_{\pm 1}^{+1},~4_{\pm 3}^{-1}.
\end{cases}
\end{equation*}}

For more details on discriminant forms, see \cite{CS1}, \cite{Ni1}, \cite{Sch1} or \cite{Str1}.

\subsection{Metaplectic covers.} \label{meta} Throughout this paper, unless otherwise stated, {$k$ is an odd integer}. Let $\mathrm{Mp}_{2}^{+}(\mathbb{R})$ be the metaplectic cover of $\mathrm{GL}_{2}^{+}(\mathbb{R})$. The elements of $\mathrm{Mp}_{2}^{+}(\mathbb{R})$ are pairs $(A,\phi)$ where $\phi$ is a holomorphic function on $\mathbb{H}$ and
$$A=\left(\begin{array}{cc}a & b \\ c & d\end{array}\right)\in\mathrm{GL}_{2}^{+}(\mathbb{R}),\quad\phi(\tau)=tj(A,\tau),~\textnormal{for some}~t\in\mathbb{C},~|t|=1.$$ 
Here $j(A,\tau)=\det(A)^{-\frac{1}{4}}(c\tau+d)^{\frac{1}{2}}$. The product of two elements $(A_{1},\phi_{1})$ and $(A_{2},\phi_{2})$ is defined by
$$(A_{1}A_{2},\phi_{1}(A_{2}\tau)\phi_{2}(\tau)).$$
To introduce the theta multiplier system, we first extend the Jacobi symbol. For an integer $c$ and an odd integer $d\neq 0$, the ``extended Jacobi symbol" $\left(\frac{c}{d}\right)$ is defined as follows.
\begin{enumerate}
\item[(1)] $\left(\frac{c}{d}\right)=0$ if $\gcd{(c,d)}>1$;
\item[(2)] $\left(\frac{0}{\pm 1}\right)=1$;
\item[(3)] If $d>0$, then $\left(\frac{c}{d}\right)$ is the usual Jacobi symbol;
\item[(4)] If $d<0$, then $\left(\frac{c}{d}\right)=\mathrm{sgn}(c)\left(\frac{c}{|d|}\right)$.
\end{enumerate}
Next we define $\varepsilon_{d}$ for odd $d$ by:
\begin{equation*}
\varepsilon_{d}=
\begin{cases}
1 & \textnormal{if}\quad d\equiv 1\pmod{4};\\
i & \textnormal{if}\quad d\equiv 3\pmod{4}.
\end{cases}
\end{equation*}
We finally define the theta multiplier system $\nu$ on $\Gamma_{0}(4)$:
$$\nu(A)=\left(\frac{c}{d}\right)\varepsilon_{d}^{-1},\quad A=\left(\begin{array}{cc}a & b \\ c & d\end{array}\right)\in\Gamma_{0}(4).$$
Note that
$$\overline{\nu}(A)=\nu^{3}(A)=\left(\frac{-1}{d}\right)\nu(A),\quad\nu(A)\nu(A^{-1})=1,\quad A\in\Gamma_{0}(4).$$
For any $A=\left(\begin{smallmatrix}a & b \\ c & d\end{smallmatrix}\right)\in\mathrm{GL}_{2}^{+}(\mathbb{R})$, we let
$$\tilde{A}=(A,j(A,\tau))\in\mathrm{Mp}_{2}^{+}(\mathbb{R}).$$

Let $\mathrm{Mp}_{2}(\mathbb{Z})$ be the metaplectic double cover of $\mathrm{SL}_{2}(\mathbb{Z})$ inside $\mathrm{Mp}_{2}^{+}(\mathbb{R})$, consisting of pairs $(A,\phi)$ with $A=\left(\begin{smallmatrix}a & b \\ c & d\end{smallmatrix}\right)\in\mathrm{SL}_{2}(\mathbb{Z})$ and $\phi^{2}=c\tau+d$. Let $S$ and $T$ denote the standard generators of $\mathrm{SL}_{2}(\mathbb{Z})$. Then
$$\tilde{S}=\left(\left(\begin{array}{cc}0 & -1 \\ 1 & 0\end{array}\right),\sqrt{\tau}\right),\quad\tilde{T}=\left(\left(\begin{array}{cc}1 & 1 \\ 0 & 1\end{array}\right),1\right)$$
generate $\mathrm{Mp}_{2}(\mathbb{Z})$.

\subsection{Modular forms.} \label{mod} Let $(A,\phi)\in\mathrm{Mp}_{2}^{+}(\mathbb{R})$ and $f:\mathbb{H}\rightarrow\mathbb{C}$ be a function. The weight {$k/2$} slash operator is defined by
{$$(f|_{k/2}(A,\phi))(\tau)=\phi^{-k}(\tau)f(A\tau),\quad A=\left(\begin{array}{cc}a & b \\ c & d\end{array}\right).$$}

Consider the Atkin--Lehner operators on the space {$M_{k/2}^{!}(4N,\chi)$}, where $N$ is odd square-free. For any odd divisor $m$ of $N$, we choose $\gamma_{m}$ in $\mathrm{SL}_{2}(\mathbb{Z})$ such that
\begin{equation*}
\gamma_{m}\equiv
\begin{cases}
S &\mod{m^{2}},\\
I &\mod{(4N/m)^{2}},
\end{cases}
\end{equation*}
and let
\begin{equation*}
\quad\gamma_{4m}:=S\gamma_{N/m}^{-1}\equiv
\begin{cases}
S &\mod{(4m)^{2}},\\
I &\mod{(4N/m)^{2}}.
\end{cases}
\end{equation*}
We shall assume for simplicity that all of the entries of $\gamma_{m}$ are positive; this can be achieved by left and/or right multiplication by matrices in $\Gamma(16N^{2})$.

For any nonzero integer $m$, let
$$\delta_{m}=\left(\begin{array}{cc}m & 0 \\ 0 & 1\end{array}\right),\quad\widetilde{\delta_{m}}=\left(\left(\begin{array}{cc}m & 0 \\ 0 & 1\end{array}\right),m^{-\frac{1}{4}}\right).$$
For any odd positive divisor $m$ of $N$, let $W(m)=\gamma_{m}^{*}\widetilde{\delta_{m}}$. Define
$$\tau_{4N}=(I,\sqrt{-i})\widetilde{\beta_{4N}}=\left(\left(\begin{array}{cc}0 & -1 \\ 4N & 0\end{array}\right),(4N)^{\frac{1}{4}}(-i\tau)^{\frac{1}{2}}\right),\quad\beta_{N}=\left(\begin{array}{cc}0 & -1 \\ 4N & 0\end{array}\right).$$
For each divisor $m$, even or odd, of $4N$, define $U(m)$ as follows:
{$$f|U(m)=m^{\frac{k}{4}-1}\sum_{j\ \textrm{mod} \ {m}}f|\widetilde{\delta_{m}}^{-1}\tilde{T}^{j}=m^{\frac{k}{4}-1}\sum_{j\ \textrm{mod} \ {m}}f|\widetilde{\delta_{m}^{-1}T^{j}}.$$}
Finally, we define $Y(p)$ for each odd prime divisor $p$ of $N$ by
{$$f|Y(p)=p^{1-\frac{k}{4}}f|U(p)W(p),$$}
and $Y(4)$ by
{$$f|Y(4)=4^{1-\frac{k}{4}}f|U(4)W(M)\tau_{4N}.$$}

\begin{proposition}[\cite{Shi1, Ueda1, Zhang2}]
Let {$f\in M_{k/2}^{!}(4N,\chi)$}.
\begin{enumerate}
\item[\textnormal{(1)}] {$f|\tau_{4N}\in M_{k/2}^{!}(4N,\chi\left(\frac{4N}{\cdot}\right))$} and $f|\tau_{4N}^{2}=f$.
\item[\textnormal{(2)}] For each $m\mid 4N$, {$f|U(m)\in M_{k/2}^{!}(4N,\chi\left(\frac{m}{\cdot}\right))$}.
\item[\textnormal{(3)}] For each $m\mid N$, {$f|W(m)\in M_{k/2}^{!}(4N,\chi\left(\frac{m}{\cdot}\right))$} and
{$$f|W(m)^{2}=\varepsilon_{m}^{-k}\chi_{m}(-1)\chi_{4N/m}(m)f.$$}
Moreover, if $m,m'\mid N$ and $\gcd{(m,m')}=1$, then $f|W(m)W(m')=\chi_{m'}(m)f|W(mm')$.
\item[\textnormal{(4)}] For any $m,m'\mid N$ with $\gcd{(m,m')}=1$,
$$f|W(m)U(m')=\chi_{m}(m')f|U(m')W(m)~\textnormal{and}~f|U(4)W(m)=f|W(m)U(4).$$
\item[\textnormal{(5)}] For any $m\mid N$, $f|\tau_{4N}U(m)W(m)=\chi_{m}(M/m)f|W(m)U(m)\tau_{4N}$.
\end{enumerate}
\end{proposition}

Now we consider the operators $Y(p)$ and $Y(4)$.

\begin{proposition}[\cite{Koh1, Ueda1, Zhang2}] \hfill
\begin{enumerate} 
\item[\textnormal{(1)}] The space {$M_{k/2}^{!}(4N,\chi)$} decomposes under $Y(4)$ into eigenspaces
{$$M_{k/2}^{!}(4N,\chi)=M_{k/2}^{!}(4N,\chi)_{\mu_{2}^{+}}\oplus M_{k/2}^{!}(4N,\chi)_{\mu_{2}^{-}},$$}
where the eigenvalues are
{$$\mu_{2}^{+}=\chi_{2}(-1)^{\frac{k}{2}+1}(-1)^{\lfloor\frac{k+1}{4}\rfloor}2^{\frac{3}{2}} \quad \text{ and } \quad \mu_{2}^{-}=-2^{-1}\mu_{2}^{+}.$$}
Moreover, {$f=\sum_{n}a(n)q^{n}\in M_{k/2}^{!}(4N,\chi)_{\mu_{2}^{+}}$} if and only if
{$$a(n)=0~\text{whenever}~\chi_{2}(-1)(-1)^{\frac{k-1}{2}}n\equiv 2,3\pmod{4}.$$}
\item[\textnormal{(2)}] Assume that $p\mid N$ with $\chi_{p}=1$. Then the space {$M_{k/2}^{!}(4N,\chi)$} decomposes under $Y(p)$ into eigenspaces
$$M_{k/2}^{!}(4N,\chi)=M_{k/2}^{!}(4N,\chi)_{\mu_{p}^{+}}\oplus M_{k/2}^{!}(4N,\chi)_{\mu_{p}^{-}},$$
where the eigenvalues are $\mu_{p}^{+}=\varepsilon_{p}^{-1}p^{\frac{1}{2}}$ and $\mu_{p}^{-}=-\mu_{p}^{+}$. Moreover, {$f=\sum_{n}a(n)q^{n}\in M_{k/2}^{!}(4N,\chi)_{\mu_{p}^{+}}$} (resp. {$M_{k/2}^{!}(4N,\chi)_{\mu_{p}^{-}}$}) if and only if
$$a(n)=0~\text{whenever}~\left(\frac{n}{p}\right)=- 1   \quad ~(resp.~ \left(\frac{n}{p}\right)= 1~)~.$$
\item[\textnormal{(3)}] Assume that $p\mid N$ with $\chi_{p}=\left(\frac{\cdot}{p}\right)$. Then the space {$M_{k/2}^{!}(4N,\chi)$} decomposes under $Y(p)$ into eigenspaces
{$$M_{k/2}^{!}(4N,\chi)=M_{k/2}^{!}(4N,\chi)_{\mu_{p}^{+}}\oplus M_{k/2}^{!}(4N,\chi)_{\mu_{p}^{-}},$$}
where the eigenvalues are $\mu_{p}^{+}=-1$ and $\mu_{p}^{-}=-p$.
\item[\textnormal{(4)}] The space {$M_{k/2}^{!}(4N,\chi)$} decomposes into the direct sum of simultaneous eigenspaces for the operators $Y(4)$ and $Y(p)$.
\end{enumerate}
\end{proposition}

\subsection{$\epsilon$-condition and reduced forms} 
\label{sec-epsilon} Let {$\mathcal{D}$} be a transitive discriminant form of odd signature $r$ and level $M$. Assume that {$\mathcal{D}_{2}=2_{\pm 1}^{+1}$}. Then $M=4N$ for some odd square-free $N$ and
{$$\chi_{2}(-1)=e_{4}(r-t)=\left(\frac{-1}{|\mathcal{D}|}\right),$$}
where $e_{4}(x)=e^{2\pi ix/4}$. Recall that {$\mathcal{D}$} determines an even Dirichlet character $\chi$ modulo $4N$. We define a sign vector $\epsilon=(\epsilon_{p})_{p}$ over $p=2$ or $p\mid N$ with $\chi_{p}\neq 1$ as follows:
{\begin{equation*}
\epsilon_{p}=
\begin{cases}
\chi_{p}(2N/p) & \textnormal{if}~p\mid N,~\chi_{p}\neq 1~\textnormal{and}~\mathcal{D}_{p}=p^{+1},\\
-\chi_{p}(2N/p) & \textnormal{if}~p\mid N,~\chi_{p}\neq 1~\textnormal{and}~\mathcal{D}_{p}=p^{-1},\\
\left(\frac{-1}{N}\right) & \textnormal{if}~p=2~\textnormal{and}~\mathcal{D}_{2}=2_{+1}^{+1},\\
-\left(\frac{-1}{N}\right) & \textnormal{if}~p=2~\textnormal{and}~\mathcal{D}_{2}=2_{-1}^{+1}.
\end{cases}
\end{equation*}} 
Therefore, {$\mathcal{D}$} determines $(4N,\chi\left(\frac{4N}{\cdot}\right),\epsilon)$: $\chi$ an even Dirichlet character modulo $4N$ and $\epsilon$ a sign vector. We shall denote $\chi'=\chi\left(\frac{4N}{\cdot}\right)$ from now on.

Given any data $(4N,\chi',\epsilon)$ with even $\chi$ and $\epsilon_{p}=\pm 1$ for $p=2$ or $p\mid N$ with $\chi_{p}\neq 1$, we define the associated modular form space {$M_{k/2}^{!\epsilon}(N,\chi')$} to be the common eigenspace with eigenvalues $\mu_{2}^{+}$ for $Y(4)$, $\mu_{p}^{\epsilon_{p}}$ for $Y(p)$ if $\chi_{p}\neq 1$ and $\mu_{p}^{+}=-1$ for $Y(p)$ if $\chi_{p}=1$. Since
\begin{align*}
\epsilon_{2}\chi_{2}'(-1)(-1)^{\frac{k-1}{2}}&=t\chi_{2}(-1)e_{4}(k-1)\\
&=\chi_{2}(-1)e_{4}(k-1)e_{4}(1-t)\\
&=\chi_{2}(-1)e_{4}(r-t)=1,
\end{align*} 
we have {$f=\sum_{n}a(n)q^{n}\in M_{k/2}^{!\epsilon}(N,\chi')$} if and only if the following two conditions are satisfied:

(1) $a(n)=0$ whenever $n\equiv 2,-\epsilon_{2}\pmod{4}$ or $\left(\frac{n}{p}\right)=-\epsilon_{p}$ for some $p\mid N$ with $\chi_{p}\neq 1$,

(2) {$f|_{k/2}Y(p)=-f$} for every $p\mid N$ with $\chi_{p}=1$.\\ 

Recall the even lattice $L$ introduced in Section \ref{intro}. Then the discriminant form {$\mathcal{D}=L'/L\cong\mathbb{Z}/2N\mathbb{Z}$} with {$\mathcal{D}=\prod_{p\mid 2N}\mathcal{D}_{p}$} is given by
{$$\mathcal{D}_{2}=2_{t}^{+1},\quad t=\left(\frac{-1}{N}\right),\quad \mathcal{D}_{p}=p^{\delta_{p}},\quad \delta_{p}=\left(\frac{2N/p}{p}\right)\quad\textnormal{for}\quad p\mid N.$$}
It follows that for such {$\mathcal{D}$}, we have $\epsilon_{p}=+1$ for all $p\mid N$ and $\chi'=\chi\left(\frac{4N}{\cdot}\right)=1$. From the above observation, we have {$f=\sum_{n}a(n)q^{n}\in M_{k/2}^{!\epsilon}(N,1)$ if and only if $a(n)=0$ unless $(-1)^{\frac{k-1}{2}}n$ is a square modulo $4N$}. Thus the space {$M_{k/2}^{!\epsilon}(N,1)$} is exactly the  same as the space {$M_{k/2}^{!+\ldots+}(N)$}. Eichler and Zagier denote the space {$M_{k/2}^{\epsilon}(N,1)$} by $M_{k}^{+\ldots+}(N)$ in \cite[p.69]{EiZa1}.

\medskip

Now we shall assume that $\chi_{p}\neq 1$ for each $p\mid N$, so $\chi'=1$. A form {$f\in M_{k/2}^{!\epsilon}(N,\chi')$} is called \textit{reduced} if $f=\frac{1}{s(m)}q^{m}+ \sum_{\ell \ge m+1}  a(\ell) q^\ell$ for some integer $m$ and if for each $n>m$ with  $a(n)\neq 0$, there does not exist {$g\in M_{k/2}^{!\epsilon}(N,\chi')$} such that $g=q^{n}+O(q^{n+1})$. Here, {$s(m)=\displaystyle\prod_{p\mid\gcd{(N,m)}}\left(1+\frac{p}{|\mathcal{D}_p|}\right)$}. If a reduced form exists for some $m$, it is unique and $\chi_{p}(m)\neq-\epsilon_{p}$ for each $p\mid N$; we denote it by $f_{m}$. The set of reduced modular forms is a basis for {$M_{k/2}^{!\epsilon}(N,\chi')$}.

The following proposition determines $m<0$ for which $f_{m}$ exists. To state it, we need some notation. Let {$\mathcal{D}^{*}$} be the dual discriminant form of {$\mathcal{D}$} given by the same abelian group with the quadratic form $-Q$. It is known that {$\mathcal{D}^{*}$} is also transitive and the corresponding data is $(4N,\chi',\epsilon^{*})$ with $\epsilon_{p}^{*}=\chi_{p}(-1)\epsilon_{p}$.

\begin{proposition}[\cite{Zhang2}, Proposition 6.1] \label{prop-zhang2}
Let {$B^{*}=\{m : f_{m}^{*}\in M_{2-k/2}^{\epsilon^{*}}(N,\chi')~\textnormal{exists}\}$}. Then for any $m<0$ with $\chi_{p}(m)\neq -\epsilon_{p}$ for all $p\mid N$, the reduced form {$f_{m}\in M_{k/2}^{!\epsilon}(N,\chi')$} exists if and only if $-m\notin B^{*}$.
\end{proposition}

\section{Proofs of Proposition \ref{prop-1.1} and Theorem \ref{thm-1.2}}
In this section, we prove the rationality of Fourier coefficients of reduced forms following the lines in \cite{BrBun1} and \cite{Zhang1}. For $f=\sum_{n}a(n)q^{n}$ and $\sigma\in\mathrm{Aut}(\mathbb{C})$, define $f^{\sigma}=\sum_{n}\sigma(a(n))q^{n}$.

\begin{lemma} \label{lem-1}
Let $\chi$ be a Dirichlet character modulo $N$ with values in $\mathbb{Q}$. If $f\in M_{k/2}^{!}(4N,\chi)$, so is $f^{\sigma}$.
\end{lemma}

\begin{proof}
It is known that $\mathrm{Aut}(\mathbb{C})$ acts on the space {$M_{k/2}(\Gamma_{1}(4N),\chi)$}, the space of holomorphic modular forms of weight {$k/2$} on $\Gamma_{1}(4N)$ with Nebentypus $\chi$, and {$\sigma(M_{k/2}(4N,\chi))=M_{k/2}(4N,\chi^{\sigma})$}. (See \cite{SeSt1}.) Since $\chi$ has values in $\mathbb{Q}$, $\mathrm{Aut}(\mathbb{C})$ acts on {$M_{k/2}(4N,\chi)$}.

Note that {$f\Delta^{k'}\in M_{k/2+12k'}(4N,\chi)$} for a sufficiently large positive integer $k'$. Here $\Delta$ is the unique normalized cusp form of weight $12$ for $\mathrm{SL}_{2}(\mathbb{Z})$. The above observation shows that {$(f\Delta^{k'})^{\sigma}\in M_{k/2+12k'}(4N,\chi)$}. But $\Delta$ has integral Fourier coefficients, hence {$f^{\sigma}\Delta^{k'}=(f\Delta^{k'})^{\sigma}\in M_{k/2+12k'}(4N,\chi)$} and {$f^{\sigma}\in M_{k/2}^{!}(4N,\chi)$}.
\end{proof}

\begin{proposition}
Let $k<0$ and let {$f=\sum_{n}a(n)q^{n}\in M_{k/2}^{!\epsilon}(N,\chi')$}. Suppose that $a(n)\in\mathbb{Q}$ for $n<0$. Then all the coefficients $a(n)$ are rational with bounded denominators.
\end{proposition}

\begin{proof}
Let $\sigma\in\mathrm{Aut}(\mathbb{C})$. By Lemma \ref{lem-1}, {$f^{\sigma}\in M_{k/2}^{!}(4N,\chi')$}. It is easy to check that the action of $\mathrm{Aut}(\mathbb{C})$ preserves the $\epsilon$-condition. Since $a(n)\in\mathbb{Q}$ for $n<0$, $h:=f-f^{\sigma}$ is holomorphic at $\infty$. By \cite[Corollary 5.5]{Zhang2}, {$h\in M_{k/2}^{\epsilon}(N,\chi')$}. But $k<0$, so $h=0$. It follows that $f$ has rational coefficients.

We know that {$\theta f\Delta^{k'}\in S_{(k+1)/2+12k'}(4N,\chi')\subset S_{(k+1)/2+12k'}(\Gamma_{1}(4N))$} for a sufficiently large positive integer $k'$. Shimura proved that {$S_{(k+1)/2+12k'}(\Gamma_{1}(4N))$} has a basis $\mathcal{B}$ consisting of forms whose Fourier coefficients at $\infty$ are rational integers. (See \cite[Theorem 3.52]{Shi2}.) Let {$S_{(k+1)/2+12k'}^{\mathbb{Q}}(\Gamma_{1}(4N))$} be the $\mathbb{Q}$-vector space of cusp forms in {$S_{(k+1)/2+12k'}(\Gamma_{1}(4N))$} whose Fourier coefficients at $\infty$ are rational numbers. Then $\mathcal{B}$ is a $\mathbb{Q}$-basis of {$S_{(k+1)/2+12k'}^{\mathbb{Q}}(\Gamma_{1}(4N))$} and {$f\theta\Delta^{k'}\in S_{(k+1)/2+12k'}^{\mathbb{Q}}(\Gamma_{1}(4N))$}. This implies that $f\theta\Delta^{k'}$ has coefficients with bounded denominators, and we conclude that the $a(n)$ are rational with bounded denominators.
\end{proof}

We are interested in integrality of Fourier coefficients. So we generalize Sturm's theorem to {$M_{k/2}^{!\epsilon}(N,\chi')$}. We begin with introducing the original Sturm's theorem.

\begin{theorem}[\cite{St1}] \label{thm-St}
Let $\mathcal{O}_{F}$ be the ring of integers of a number field $F$, $\mathfrak{p}$  any prime ideal, $N'$ a positive integer and $k'$  a positive integer. Assume {$f=\sum_{n}a(n)q^{n}\in M_{k'}(N',\chi)\cap\mathcal{O}_{F}[\![ q]\!]$}. If $a(n)\in\mathfrak{p}$ for $n\leq\frac{k'}{12}[\mathrm{SL}_{2}(\mathbb{Z}):\Gamma_{0}(N')]$, then $a(n)\in\mathfrak{p}$ for all $n$.
\end{theorem}

Using Theorem \ref{thm-St}, Kim, Lee and Zhang proved the following:

\begin{corollary}[\cite{KLZ1}, Corollary 3.2] \label{cor-KLZ}
Let $k'$ be a positive integer. Assume {$f=\sum_{n}a(n)q^{n}\in M_{k'}(4N,\chi)\cap\mathbb{Q}[\![ q]\!]$} with bounded denominator. If $a(n)\in\mathbb{Z}$ for $n\leq\frac{k'}{12}[\mathrm{SL}_{2}(\mathbb{Z}):\Gamma_{0}(4N)]$, then $a(n) \in\mathbb{Z}$ for all $n$.
\end{corollary}

We extend this result to half-integral weight case. Let
$$\theta(\tau)=\sum_{n\in\mathbb{Z}}q^{n^{2}}=1+2q+2q^{4}+2q^{9}+\cdots,\quad (q=e^{2\pi i\tau},~\tau\in\mathbb{H}).$$
\begin{corollary} \label{cor-let}
{Let $k>0$ be an odd integer} and assume that {$f=\sum_{n}a(n)q^{n}\in M_{k/2}(4N,\chi)\cap\mathbb{Q}[\![ q]\!]$} with bounded denominator. If $a(n)\in\mathbb{Z}$ for {$n\leq\frac{k}{12}[\mathrm{SL}_{2}(\mathbb{Z}):\Gamma_{0}(4N)]$}, then $a(n)\in\mathbb{Z}$ for all $n$.
\end{corollary}

\begin{proof}
By multiplying {$\theta^{k}$}, we have {$f\theta^{k}\in M_{k}(4N,\chi)$}. It suffices to show that all the coefficients of {$f\theta^{k}$} are integers. Since $a(n)\in\mathbb{Z}$ for $n\leq\frac{k}{12}[\mathrm{SL}_{2}(\mathbb{Z}):\Gamma_{0}(4N)]$, the same thing holds for the coefficients of {$f\theta^{k}$}. By Corollary \ref{cor-KLZ}, every coefficient of {$f\theta^{k}$} is an integer.
\end{proof}

Now we are ready to prove Proposition \ref{prop-1.1} and Theorem \ref{thm-1.2}.

\begin{proof}[Proof of Proposition \ref{prop-1.1}]
{Since $k'\geq |\mathrm{ord}_{\infty}(f)|/4N$, we see that $f(\tau)\Delta(4N\tau)^{k'}\in M_{k/2+12k'}^{\epsilon}(N,\chi)$ and that every coefficient of $f(\tau)\Delta(4N\tau)^{k'}$ less than or equal to $\frac{k+12k'}{12}[\mathrm{SL}_{2}(\mathbb{Z}):\Gamma_{0}(4N)]$ is an integer. By Corollary \ref{cor-let}, $f(\tau)\Delta(4N\tau)^{k'}$ has integer Fourier coefficients, hence so does $f$.}
\end{proof}

%\begin{lemma} \label{lem-ab}
%Let $m_{\epsilon}$ denote the maximal $m$ such that $f_{m}^{*}$ exists. Assume that for all $n\in\mathbb{Z}$ and $m\geq -N-m_{\epsilon}$, we have $s(n)a(m,n)\in\mathbb{Z}$. Then $s(n)a(m,n)\in\mathbb{Z}$ for all $m,n\in\mathbb{Z}$.
%\end{lemma}

%\begin{proof}
%Consider any reduced form $f_{m'}$ with $m'<-N-m_{\epsilon}$. There exists integers $-N-m_{\epsilon}\leq m_{0}'<-m_{\epsilon}$ and $l\geq 1$ such that $m'=-Nl+m_{0}'$. By maximality of $m_{\epsilon}$, $f_{m_{0}'}$ exists. Consider now
%$$g=j(N\tau)^{l}f_{m_{0}'}=\sum_{n}b(n)q^{n}\in A^{\epsilon}(N,k,\chi'),$$
%where $j(\tau)$ denotes the classical $j$-function. It is known that $j$ has integral Fourier coefficients. By the assumption on $f_{m_{0}'}$,  we see that $b(n)s(n)\in\mathbb{Z}$ for each $n$.\\
%\indent Now $g$ and $f_{m'}$ share the same lowest power term, and we must have that
%$$f_{m'}=g-\sum_{m>m'}s(m)b(m)f_{m}.$$
%Hence $s(n)a_{m'}(m)=s(n)b(n)-s(m)b(m)s(n)a_{m}(n)\in\mathbb{Z}$ by the assumption and induction on $m$.
%\end{proof}

\begin{proof}[Proof of Theorem \ref{thm-1.2}]
Consider any reduced form $f_{m'}$ with $m'<-4N-m_{\epsilon}$. There exist integers $-4N-m_{\epsilon}\leq m_{0}'<-m_{\epsilon}$ and $l\geq 1$ such that $m'=-4Nl+m_{0}'$. By maximality of $m_{\epsilon}$, $f_{m_{0}'}$ exists. Consider now
$$g=j(4N\tau)^{l}f_{m_{0}'}=\sum_{n}b(n)q^{n}\in M_{k/2}^{!\epsilon}(N,\chi'),$$
where $j(\tau)$ denotes the classical $j$-function. It is known that $j$ has integral Fourier coefficients. By the assumption on $f_{m_{0}'}$,  we see that $b(n)s(m_{0}')\in\mathbb{Z}$ for each $n$.

Now $s(m_{0}')g$ and $s(m')f_{m'}$ share the same lowest power term, and we must have that
$$s(m')f_{m'}=s(m_{0}')g-\sum_{m>m'}s(m_{0}')b(m)s(m)f_{m}.$$
Hence $s(m')a_{m'}(n)=s(m_{0}')b(n)-s(m)b(m)s(m_{0}')a_{m}(n)\in\mathbb{Z}$ by the assumption and induction on $m$.
\end{proof}

\section{Proofs of Theorem \ref{thm-1.4} and Corollary \ref{cor-1.5}}
From now on, we shall assume that, for any reduced form
$$f_{m}=\sum_{n}a(m,n)q^{n}\in M_{k/2}^{!\epsilon}(N,\chi),$$
the modular form $s(m)f_{m}$ has integral Fourier coefficients. We remark that such integrality for each fixed reduced form can be verified by Proposition \ref{prop-1.1}. {Also, as we showed in Example \ref{ex-int}, every reduced form in the space {$M_{1/2}^{!+\cdots +}(7,1)$} satisfies the assumption.

We begin with a lemma.

\begin{lemma} \label{lem-sm}
Let $N\geq 1$ be a square-free integer. Then we have {$M_{3/2}^{+\cdots +}(N,1)=S_{3/2}^{+\cdots +}(N,1)$}.
\end{lemma}

\begin{proof}
Let {$f=\sum_{n\geq 0}a(n)q^{n}\in M_{3/2}^{+\cdots +}(N,1)$}. By Borcherds' obstruction theorem (Theorem 3.1 of \cite{Bor1}), we get
$$s(0)a(0)b(0)=0$$
for each {$g=\sum_{n}b(n)q^{n}\in M_{1/2}^{\epsilon^{*}}(N,1)$}. Since $N$ is square-free,
{$$M_{1/2}^{\epsilon^{*}}(N,1)=\mathbb{C}\theta.$$}
Setting $g=\theta$, we obtain
$$s(0)a(0)=0.$$
Since $s(0)\neq 0$, the form $f$ vanishes at $\infty$. By \cite[Proposition 5.3]{Zhang2}, the vector-valued form $\psi(f)$ is a cusp form. Here $\psi$ is the map constructed in Chapter 5 of \cite{Zhang2}. We conclude from \cite[Corollary 5.5]{Zhang2} that {$f=\psi^{-1}(\psi(f))\in S_{3/2}^{+\cdots +}(N,1)$}.
\end{proof}

\begin{remark} \label{rmk-zero}
If $k\geq 3$ is an odd integer and $(k,\epsilon)\neq (3,+)$, then
{$$M_{2-k/2}^{\epsilon^{*}}(N,1)=0.$$}
\end{remark}

Let $N\geq 1$ be an odd square-free integer. Suppose that $\ell$ is a prime with $\ell\nmid 4N$. 
Consider a modular form
$$f(\tau)=\sum_{n} a(n)q^{n}\in M_{k/2}^{!\epsilon}(N,1),$$ where the sum is over $n$ such that $\chi_{p}(n)\neq -\epsilon_{p}$ for all $p\mid N$.
Then the action of the Hecke operator $T_{k/2,4N}(\ell^{2})$ on $f$
is given by
\begin{equation}\label{1}
f(\tau)\mid T_{k/2,4N}(\ell^{2})=\sum_{n}\left(a(\ell^{2}n)+\ell^{\lambda-1}\left(\frac{(-1)^{\lambda}n}{\ell}\right)a(n)+\ell^{2\lambda-1}a(n/\ell^{2})\right)q^{n},
\end{equation}
where $\lambda=(k-1)/2$ and the sum is over $n$ such that $\chi_{p}(n)\neq -\epsilon_{p}$ for all $p\mid N$. Here we set $a(n/\ell^{2})=0$ if $\ell^{2}\nmid n$. Note that {$f(\tau)\mid T_{k/2,4N}(\ell^{2})\in M_{k/2}^{!\epsilon}(N,1)$}. We define $T_{k/2,4N}(\ell^{2n})$ for $n\geq 2$ recursively by
$$T_{k/2,4N}(\ell^{2n}):=T_{k/2,4N}(\ell^{2n-2})T_{k/2,4N}(\ell^{2})-\ell^{k-2}T_{k/2,4N}(\ell^{2n-4}).$$

\begin{remark} For $n\geq 2$, our $T_{k/2,4N}(\ell^{2n})$ is different from the $\ell^{2n}$-th Hecke operator given in \cite{Shi1}. See \cite[p.241]{Pur1} for details.
\end{remark}

By Proposition \ref{prop-zhang2}, the reduced form {$f_{m}\in M_{k/2}^{!\epsilon}(N,1)$} exists for every $m<0$ with $\chi_{p}(m)\neq -\epsilon_{p}$ for all $p\mid N$. We write
$$F_{m}(\tau)=s(m)f_{m}(\tau)=q^{m}+\sum\limits_{\substack{d\geq 0 \\ \chi_{p}(d)\neq -\epsilon_{p}~\textnormal{for all}~p\mid N}}B(m,d)q^{d}.$$
For any positive integer $t$ with $\gcd{(t,4N)}=1$, define
$$F_{m}^{(t)}:=F_{m}\mid T(t^{2}).$$
Then we obtain the coefficients $B_{t}(m,d)$ from the equation
$$F_{m}^{(t)}(\tau)=\mbox{(principal part)}+\sum\limits_{\substack{d\geq 0,\\ \chi_{p}(d)\neq-\epsilon_{p}~\textnormal{for all}~p\mid N}}B_{t}(m,d)q^{d}.$$

\medskip

For the rest of this section, let $k\geq 3$  be an odd integer and set $\lambda=(k-1)/2$, and assume 
% \begin{centering}
\begin{enumerate}
\item $m \in \mathbb Z_{<0}$ such that  $\chi_{p}(m)\neq -\epsilon_{p}$ for all $p\mid N$,
\item $\ell$ is a prime with $\ell\nmid 4N$ and $\ell^{2}\nmid m$.
\end{enumerate}
% \end{centering}

\begin{proposition} \label{prop-FG}
 Assume that {$S_{k/2}^{\epsilon}(N,1)=0$}. Then, for any positive integer $t$ with $(t,4N)=1$ and any positive integer $n$, we have
$$F_{m}^{(t)}|T_{k/2,4N}(\ell^{2n})-\ell^{\lambda-1}\left(\frac{(-1)^{\lambda}m}{\ell}\right)F_{m}^{(t)}|T_{k/2,4N}(\ell^{2n-2})=\ell^{(k-2)n}F_{\ell^{2n}m}^{(t)}.$$
\end{proposition}

\begin{proof}

For convenience, define $G_{0}^{(t)}:=F_{m}^{(t)}$, and, for each $n\geq 1$, 
\begin{equation}\label{2}
G_{n}^{(t)}:=F_{m}^{(t)}\mid T_{k/2,4N}(\ell^{2n})-\ell^{\lambda-1}\left(\frac{(-1)^{\lambda}m}{\ell}\right)F_{m}^{(t)}\mid T_{k/2,4N}(\ell^{2n-2}).
\end{equation}
We need to show $G_{n}^{(t)} = \ell^{(k-2)n}F_{\ell^{2n}m}^{(t)}$.
Since the Hecke operators commute, it suffices to prove the proposition in the case $t=1$, which we now assume.

We claim that
\begin{equation}\label{3}
G_{n}^{(1)}=G_{n-1}^{(1)}\mid T_{k/2,4N}(\ell^{2}) -\ell^{k-2}\cdot G_{n-2}^{(1)}\quad\textnormal{for}\quad n\geq 2.
\end{equation}
Indeed, if $n=2$, then
\begin{align*}
G_{2}^{(1)}-&G_{1}^{(1)}\mid T_{k/2,4N}(\ell^{2})+\ell^{k-2}\cdot G_{0}^{(1)}\\
&=\left(F_{m}^{(1)}\mid T_{k/2,4N}(\ell^{4})-\ell^{\lambda-1}\left(\frac{(-1)^{\lambda}m}{\ell}\right)F_{m}^{(1)}\mid T_{k/2,4N}(\ell^{2})\right)\\
&\quad-\left(\left(F_{m}^{(1)}\mid T_{k/2,4N}(\ell^{2})\right)\mid T_{k/2,4N}(\ell^{2})-\ell^{\lambda-1}\left(\frac{(-1)^{\lambda}m}{\ell}\right)F_{m}^{(1)}\mid T_{k/2,4N}(\ell^{2})\right)+\ell^{k-2}\cdot F_{m}^{(1)}\\
&=F_{m}^{(1)}\mid T_{k/2,4N}(\ell^{4})-\left(F_{m}^{(1)}\mid T_{k/2,4N}(\ell^{2})\right)\mid T_{k/2,4N}(\ell^{2})+\ell^{k-2}\cdot F_{m}^{(1)} =0 .\\
\end{align*}
For $n\geq 3$,
\begin{align*}
G_{n-1}^{(1)}\mid &T_{k/2,4N}(\ell^{2})-\ell^{k-2}\cdot G_{n-2}^{(1)}\\
&=\left(F_{m}^{(1)}\mid T_{k/2,4N}(\ell^{2n-2})\right)\mid T_{k/2,4N}(\ell^{2})-\ell^{\lambda-1}\left(\frac{(-1)^{\lambda}m}{\ell}\right)\cdot\left(F_{m}^{(1)}\mid T_{k/2,4N}(\ell^{2n-4})\right)\mid T_{k/2,4N}(\ell^{2})\\
&\quad-\ell^{k-2}\cdot\left(F_{m}^{(1)}\mid T_{k/2,4N}(\ell^{2n-4})-\ell^{\lambda-1}\left(\frac{(-1)^{\lambda}m}{\ell}\right)F_{m}^{(1)}\mid T_{k/2,4N}(\ell^{2n-6})\right)\\
&=\left(F_{m}^{(1)}\mid T_{k/2,4N}(\ell^{2n-2})T(\ell^{2})-\ell^{k-2}\cdot F_{m}^{(1)}\mid T_{k/2,4N}(\ell^{2n-4})\right)\\
&\quad-\ell^{\lambda-1}\left(\frac{(-1)^{\lambda}m}{\ell}\right)\cdot\left(F_{m}^{(1)}\mid T_{k/2,4N}(\ell^{2n-4})T(\ell^{2})-\ell^{k-2}\cdot F_{m}^{(1)}\mid T_{k/2,4N}(\ell^{2n-6})\right)\\
&=F_{m}^{(1)}\mid T_{k/2,4N}(\ell^{2n})-\ell^{\lambda-1}\left(\frac{(-1)^{\lambda}m}{\ell}\right)\cdot F_{m}^{(t)}\mid T_{k/2,4N}(\ell^{2n-2}) =G_{n}^{(1)}. \\
\end{align*}

Since $G_{0}^{(1)}=F_{m}^{(1)}=F_{m}$, the principal part of $G_{0}^{(1)}$ is $q^{m}$. By \eqref{1}, the principal part of $G_{1}^{(1)}$ is $\ell^{k-2} q^{m\ell^{2}}$. Moreover, we see from \eqref{3} that, for all $n\geq 0$, the principal part of $G_{n}^{(1)}$ is equal to $\ell^{(k-2)n}q^{m\ell^{2n}}$. Since  $F_{m\ell^{2n}}^{(1)}=F_{m\ell^{2n}}$ has principal part $q^{m\ell^{2n}}$, {$G_n^{(1)} - \ell^{(k-2)n}F_{\ell^{2n}m}$ is holomorphic at the cusp $\infty$. Arguing as in the proof of Lemma \ref{lem-sm}}, we have 
{\[ G_n^{(1)} - \ell^{(k-2)n}F_{\ell^{2n}m} \in M_{k/2}^{\epsilon}(N,1).
\]}
If $(k,\epsilon)=(3,+)$, then it follows from Lemma \ref{lem-sm} that
{$$G_n^{(1)} - \ell^{(k-2)n}F_{\ell^{2n}m}  \in S_{k/2}^{\epsilon}(N,1).$$} Since {$S_{k/2}^{\epsilon}(N,1) =\{ 0 \}$} by assumption, we have $G_n^{(1)} = \ell^{(k-2)n}F_{\ell^{2n}m}$.

If $(k,\epsilon) \neq (3,+)$, then {$M_{2-k/2}^{\epsilon^{*}}(N,1)=0$} (Remark \ref{rmk-zero}). By Borcherds' obstruction theorem, there exists a reduced form $g$ such that $g=1+O(q)$. We see from the definition of reduced forms that $B(m,0)=B(\ell^{2n}m,0)=0$. Hence the constant term of $G_{n}^{(1)}-\ell^{(k-2)n}F_{\ell^{2n}m}$ is zero, and thus
{$$G_n^{(1)} - \ell^{(k-2)n}F_{\ell^{2n}m} \in S_{k/2}^{\epsilon}(N,1)=\{0\}.$$}
Therefore we have $G_n^{(1)} = \ell^{(k-2)n}F_{\ell^{2n}m}$ in this case too.

\end{proof}

Write
$$G_{n}^{(t)}=\mbox{(principal part)}+\sum\limits_{\substack{d\geq 0, \\ \chi_{p}(d)\neq -\epsilon_{p}~\textnormal{for all}~p\mid N}}C_{n}(d)q^{d}.$$
Proposition \ref{prop-FG} implies that, for all $n$ and $d$,
\begin{equation}\label{4}
C_{n}(d)=\ell^{(k-2)n}B_{t}(\ell^{2n}m,d).
\end{equation}

\begin{lemma} \label{lem-BC}
The following are true:
\begin{enumerate}
\item[\textnormal{(\lowerromannumeral{1})}] For any $d\geq 0$ with $\chi_{p}(d)\neq -\epsilon_{p}$ for all $p\mid N$, we have
$$C_{n}(\ell^{2}d)-\ell^{k-2}\cdot C_{n-1}(d)=C_{0}(\ell^{2n+2}d)-\ell^{\lambda-1}\left(\frac{(-1)^{\lambda}m}{\ell}\right)C_{0}(\ell^{2n}d).$$
\item[\textnormal{(\lowerromannumeral{2})}] If $\chi_{p}\neq -\epsilon_{p}$ for all $p\mid N$ and $\ell\parallel d$, then
$$C_{n}(d)=C_{0}(\ell^{2n}d)-\ell^{\lambda-1}\left(\frac{(-1)^{\lambda}m}{\ell}\right)C_{0}(\ell^{2n-2}d).$$

\item[\textnormal{(\lowerromannumeral{3})}] If $\chi_{p}\neq -\epsilon_{p}$ for all $p\mid N$ and $\ell\nmid d$, then
$$C_{n}(d)=C_{0}(\ell^{2n}d)+\left[\left(\frac{(-1)^{\lambda}d}{\ell}\right)-\left(\frac{(-1)^{\lambda}m}{\ell}\right)\right]\cdot\sum_{k=1}^{n}\ell^{(\lambda-1)k}\left(\frac{(-1)^{\lambda}d}{\ell}\right)^{k-1}C_{0}(\ell^{2n-2k}d).$$
\end{enumerate}
\end{lemma}

\begin{proof}
We first prove (\lowerromannumeral{1}). Note that
\begin{align*}
\ell^{2}d&\mbox{-th coefficient of }G_{1}^{(m)}=C_{1}(\ell^{2}d),\\
\ell^{2}d&\mbox{-th coefficient of }F_{m}^{(t)}\mid T_{k/2,4N}(\ell^{2})-\ell^{\lambda-1}\left(\frac{(-1)^{\lambda}m}{\ell}\right)F_{m}^{(t)}\\
&=B_{t}(m,\ell^{4}d)+\ell^{\lambda-1}\left(\frac{(-1)^{\lambda}\ell^{2}d}{\ell}\right)B_{t}(m,\ell^{2}d)+\ell^{k-2}\cdot B_{t}(m,d)\\
&\quad-\ell^{\lambda-1}\left(\frac{(-1)^{\lambda}m}{\ell}\right)B_{t}(m,\ell^{2}d)\\
&=B_{t}(m,\ell^{4}d)+\ell^{k-2}\cdot B_{t}(m,d)-\ell^{\lambda-1}\left(\frac{(-1)^{\lambda}m}{\ell}\right)B_{t}(m,\ell^{2}d)\\
&=C_{0}(\ell^{4}d)+\ell^{k-2}\cdot C_{0}(d)-\ell^{\lambda-1}\left(\frac{(-1)^{\lambda}m}{\ell}\right)C_{0}(\ell^{2}d).
\end{align*}
By \eqref{2}, we have
$$C_{1}(\ell^{2}d)=C_{0}(\ell^{4}d)+\ell^{k-2}\cdot C_{0}(d)-\ell^{\lambda-1}\left(\frac{(-1)^{\lambda}m}{\ell}\right)C_{0}(\ell^{2}d).$$
Hence,
$$C_{1}(\ell^{2}d)-\ell^{k-2}\cdot C_{0}(d)=C_{0}(\ell^{4}d)-\ell^{\lambda-1}\left(\frac{(-1)^{\lambda}m}{\ell}\right)C_{0}(\ell^{2}d).$$
When $n\geq 2$, we use \eqref{3} to find that
$$C_{n}(\ell^{2}d)-\ell^{k-2}\cdot C_{n-1}(d)=C_{n-1}(\ell^{4}d)-\ell^{k-2}\cdot C_{n-2}(\ell^{2}d)=\cdots=C_{1}(\ell^{2n}d)-\ell^{k-2}\cdot C_{0}(\ell^{2n-2}d).$$
From \eqref{2}, we see that
$$C_{1}(\ell^{2n}d)=C_{0}(\ell^{2n+2}d)+\ell^{k-2}\cdot C_{0}(\ell^{2n-2}d)-\ell^{\lambda-1}\left(\frac{(-1)^{\lambda}m}{\ell}\right)C_{0}(\ell^{2n}d).$$
Thus we obtain
$$C_{n}(\ell^{2}d)-\ell^{k-2}\cdot C_{n-1}(d)=C_{0}(\ell^{2n+2}d)-\ell^{\lambda-1}\left(\frac{(-1)^{\lambda}m}{\ell}\right)C_{0}(\ell^{2n}d).$$

We now prove (\lowerromannumeral{2}) and (\lowerromannumeral{3}). Observe that
\begin{align*}
d&\mbox{-th coefficient of }G_{1}^{(t)}=C_{1}(d),\\
d&\mbox{-th coefficient of }F_{m}^{(t)}\mid T_{k/2,4N}(\ell^{2})-\ell^{\lambda-1}\left(\frac{(-1)^{\lambda}m}{\ell}\right)F_{m}^{(t)}\\
&=B_{t}(m,\ell^{2}d)+\ell^{\lambda-1}\left(\frac{(-1)^{\lambda}d}{\ell}\right)B_{t}(m,d)+\ell^{k-2}\cdot B_{t}(m,d/\ell^{2})\\
&\quad-\ell^{\lambda-1}\left(\frac{(-1)^{\lambda}m}{\ell}\right)B_{t}(m,d)\\
&=
\begin{cases}
B_{t}(m,\ell^{2}d)-\ell^{\lambda-1}\left(\frac{(-1)^{\lambda}m}{\ell}\right)B_{t}(m,d) & \textnormal{if}\quad\ell\parallel d,\\
B_{t}(m,\ell^{2}d)+\ell^{\lambda-1}\left[\left(\frac{(-1)^{\lambda}d}{\ell}\right)-\left(\frac{(-1)^{\lambda}m}{\ell}\right)\right]B_{t}(m,d) & \textnormal{if}\quad\ell\nmid d,
\end{cases}\\
&=
\begin{cases}
C_{0}(\ell^{2}d)-\ell^{\lambda-1}\left(\frac{(-1)^{\lambda}m}{\ell}\right)C_{0}(d) & \textnormal{if}\quad\ell\parallel d,\\
C_{0}(\ell^{2}d)+\ell^{\lambda-1}\left[\left(\frac{(-1)^{\lambda}d}{\ell}\right)-\left(\frac{(-1)^{\lambda}m}{\ell}\right)\right]C_{0}(d) & \textnormal{if}\quad\ell\nmid d.
\end{cases}
\end{align*}
By \eqref{2}, we have
$$C_{1}(d)=
\begin{cases}
C_{0}(\ell^{2}d)-\ell^{\lambda-1}\left(\frac{(-1)^{\lambda}m}{\ell}\right)C_{0}(d) & \textnormal{if}\quad\ell\parallel d,\\
C_{0}(\ell^{2}d)+\ell^{\lambda-1}\left[\left(\frac{(-1)^{\lambda}d}{\ell}\right)-\left(\frac{(-1)^{\lambda}m}{\ell}\right)\right]C_{0}(d) & \textnormal{if}\quad\ell\nmid d.
\end{cases}
$$
On the other hand, it follows from \eqref{3} that
$$C_{n}(d)=C_{n-1}(\ell^{2}d)-\ell^{k-2}\, C_{n-2}(d)+\ell^{\lambda-1}\left(\frac{(-1)^{\lambda}d}{\ell}\right)C_{n-1}(d)$$
for $n\geq 2$. Applying part (\lowerromannumeral{1}) to $C_{n-1}(\ell^{2}d)-\ell^{k-2}\, C_{n-2}(d)$, we obtain
$$C_{n}(d)=C_{0}(\ell^{2n}d)-\ell^{\lambda-1}\left(\frac{(-1)^{\lambda}m}{\ell}\right)C_{0}(\ell^{2n-2}d)+\ell^{\lambda-1}\left(\frac{(-1)^{\lambda}d}{\ell}\right)C_{n-1}(d).$$
If $\ell\parallel d$, then we immediately obtain part (\lowerromannumeral{2}). Now assume that $\ell\nmid d$. Then by induction we have
\begin{align*}
C_{n}(d)&=C_{0}(\ell^{2n}d)-\ell^{\lambda-1}\left(\frac{(-1)^{\lambda}m}{\ell}\right)C_{0}(\ell^{2n-2}d)+\ell^{\lambda-1}\left(\frac{(-1)^{\lambda}d}{\ell}\right)C_{n-1}(d)\\
&=C_{0}(\ell^{2n}d)-\ell^{\lambda-1}\left(\frac{(-1)^{\lambda}m}{\ell}\right)C_{0}(\ell^{2n-2}d)+\ell^{\lambda-1}\left(\frac{(-1)^{\lambda}d}{\ell}\right)C_{0}(\ell^{2n-2}d)\\
&\quad+\ell^{\lambda-1}\left(\frac{(-1)^{\lambda}d}{\ell}\right)\cdot\left[\left(\frac{(-1)^{\lambda}d}{\ell}\right)-\left(\frac{(-1)^{\lambda}m}{\ell}\right)\right] \cdot \sum_{k=1}^{n-1}\ell^{(\lambda-1)k}\left(\frac{(-1)^{\lambda}d}{\ell}\right)^{k-1}C_{0}(\ell^{2n-2k-2}d)\\
&=C_{0}(\ell^{2n}d)+\ell^{\lambda-1}\left[\left(\frac{(-1)^{\lambda}d}{\ell}\right)-\left(\frac{(-1)^{\lambda}m}{\ell}\right)\right]\cdot C_{0}(\ell^{2n-2}d)\\
&\quad+\left[\left(\frac{(-1)^{\lambda}d}{\ell}\right)-\left(\frac{(-1)^{\lambda}m}{\ell}\right)\right]\cdot\sum_{k=2}^{n}\ell^{(\lambda-1)k}\left(\frac{(-1)^{\lambda}d}{\ell}\right)^{k-1}C_{0}(\ell^{2n-2k}d)\\
&=C_{0}(\ell^{2n}d)+\left[\left(\frac{(-1)^{\lambda}d}{\ell}\right)-\left(\frac{(-1)^{\lambda}m}{\ell}\right)\right]\cdot\sum_{k=1}^{n}\ell^{(\lambda-1)k}\left(\frac{(-1)^{\lambda}d}{\ell}\right)^{k-1}C_{0}(\ell^{2n-2k}d).
\end{align*}
This proves the identity in part (iii).
\end{proof}

We now prove Theorem \ref{thm-1.4} and Corollary \ref{cor-1.5}.

\begin{proof}[Proof of Theorem \ref{thm-1.4}] \hfill

\noindent (\lowerromannumeral{1}) 
By \eqref{4} and Lemma \ref{lem-BC} (\lowerromannumeral{1}), we have
\begin{align*}
& B_{t}(m,\ell^{2n+2}d)-\ell^{\lambda-1}\left(\frac{(-1)^{\lambda}m}{\ell}\right)B_{t}(m,\ell^{2n}d)\\
=&C_{0}(\ell^{2n+2}d)-\ell^{\lambda-1}\left(\frac{(-1)^{\lambda}m}{\ell}\right)C_{0}(\ell^{2n}d)=C_{n}(\ell^{2}d)-\ell^{k-2}C_{n-1}(d)\\
=&\ell^{(k-2)n}B_{t}(\ell^{2n}m,\ell^{2}d)-\ell^{k-2}\cdot\ell^{(k-2)(n-1)}B_{t}(\ell^{2n-2}m,d)\\
=&\ell^{(k-2)n} \left \{ B_{t}(\ell^{2n}m,\ell^{2}d)-B_{t}(\ell^{2n-2}m,d) \right \}.
\end{align*}
(\lowerromannumeral{2}) Using \eqref{4} and Lemma \ref{lem-BC} (\lowerromannumeral{3}), we obtain
\begin{align*}
&\ell^{(k-2)n}B_{t}(\ell^{2n}m,d)=C_{n}(d)\\
=&C_{0}(\ell^{2n}d)+\left[\left(\frac{(-1)^{\lambda}d}{\ell}\right)-\left(\frac{(-1)^{\lambda}m}{\ell}\right)\right]\cdot\sum_{k=1}^{n}\ell^{(\lambda-1)k}\left(\frac{(-1)^{\lambda}d}{\ell}\right)^{k-1}C_{0}(\ell^{2n-2k}d)\\
=&B_{t}(m,\ell^{2n}d)+\left[\left(\frac{(-1)^{\lambda}d}{\ell}\right)-\left(\frac{(-1)^{\lambda}m}{\ell}\right)\right]\cdot\sum_{k=1}^{n}\ell^{(\lambda-1)k}\left(\frac{(-1)^{\lambda}d}{\ell}\right)^{k-1}B_{t}(m,\ell^{2n-2k}d).
\end{align*} 
(\lowerromannumeral{3}) By \eqref{4} and Lemma \ref{lem-BC} (\lowerromannumeral{2}),
\begin{align*}
\ell^{(k-2)n}B_{t}(\ell^{2n}m,d)&=C_{n}(d)=C_{0}(\ell^{2n}d)-\ell^{\lambda-1}\left(\frac{(-1)^{\lambda}m}{\ell}\right)C_{0}(\ell^{2n-2}d)\\
&=B_{t}(m,\ell^{2n}d)-\ell^{\lambda-1}\left(\frac{(-1)^{\lambda}m}{\ell}\right)B_{t}(m,\ell^{2n-2}d).
\end{align*}
\end{proof}

\begin{proof}[Proof of Corollary \ref{cor-1.5}]
(1) First, suppose that $\left(\frac{-d}{\ell}\right)=\left(\frac{-m}{\ell}\right)\neq 0$. Then $\ell\nmid d$. By Theorem \ref{thm-1.4} (\lowerromannumeral{2}),
$$\ell^{(k-2)n}B_{t}(\ell^{2n}m,d)=B_{t}(m,\ell^{2n}d).$$
Now assume that $\ell\parallel d$ and $\ell\parallel m$. Then by Theorem \ref{thm-1.4} (\lowerromannumeral{3}),
$$\ell^{(k-2)n}B_{t}(\ell^{2n}m,d)=B_{t}(m,\ell^{2n}d).$$

(2) If $\ell\nmid d$, then $\ell\parallel\ell d$. By Theorem \ref{thm-1.4} (\lowerromannumeral{3}),
\begin{align*}
&B_{t}(m,\ell^{2n+1}d)-\ell^{\lambda-1}\left(\frac{(-1)^{\lambda}m}{\ell}\right)B_{t}(m,\ell^{2n-1}d)\\
=&B_{t}(m,\ell^{2n}(\ell d))-\ell^{\lambda-1}\left(\frac{(-1)^{\lambda}m}{\ell}\right)B_{t}(m,\ell^{2n-2}(\ell d))\\
=&\ell^{(k-2)n}\, B_{t}(\ell^{2n}m,\ell d)\equiv 0\pmod{\ell^{(k-2)n}}.
\end{align*}
If $\ell\mid d$, then by Theorem \ref{thm-1.4} (\lowerromannumeral{1}), we obtain
\begin{align*}
&B_{t}(m,\ell^{2n+1}d)-\ell^{\lambda-1}\left(\frac{(-1)^{\lambda}m}{\ell}\right)B_{t}(m,\ell^{2n-1}d)\\
=&B_{t}(m,\ell^{2n+2}(d/\ell))-\ell^{\lambda-1}\left(\frac{(-1)^{\lambda}m}{\ell}\right)B_{t}(m,\ell^{2n}(d/\ell))\\
=&\ell^{(k-2)n}\, (B_{t}(\ell^{2n}m,\ell d)-B_{t}(\ell^{2n-2}m,d/\ell))
\equiv 0\pmod{\ell^{(k-2)n}}.
\end{align*}

\end{proof}

\section*{Acknowledgement}
We would like to thank Professor Yichao Zhang for his kind and valuable comments.

\end{document}